\documentclass{article}
\usepackage[margin=1.5in]{geometry}
\usepackage{amsmath}
\usepackage{amssymb}
\usepackage{amsthm}
\usepackage[utf8]{inputenc}

\usepackage[pdftex]{graphicx}
\usepackage{times}
\usepackage{enumerate}

\usepackage{xspace}
\usepackage{comment}
\usepackage{algorithm}
\usepackage{algorithmic}
\usepackage{eufrak}
\usepackage{color}
\definecolor{red}{rgb}{1.0,0.0,0.0}
\definecolor{blue}{rgb}{0.0,0.0,1.0}

\newcommand{\cJ}{{\cal J}}

\newcommand{\bx}{{\bf x}}

\newcommand{\bz}{{\bf z}}
\newcommand{\by}{{\bf y}}

\newcommand{\reals}{\mathbb{R}}
\newcommand{\comps}{\mathbb{C}}

\newcommand{\sfT}{\textsf{T}}

\newtheorem{thm}{Theorem}
\newtheorem{assumption}{Assumption}
\newtheorem{cor}{Corollary}

\theoremstyle{definition}
\newtheorem{rem}{Remark}
\newtheorem{definition}{Definition}

\newcommand{\diam}{\mathsf{diam}}

\newcommand{\conv}{\mathsf{ch}}

\newcommand{\proj}[2][]{\mathrm{\mathcal{P}}_{#1}(#2)}

\renewcommand{\S}{\ensuremath{\mathcal{S}}\xspace}

\newcommand{\A}{\ensuremath{\mathcal{A}}\xspace}
\newcommand{\U}{\ensuremath{\mathcal{U}}\xspace}
\newcommand{\X}{\ensuremath{\mathcal{X}}\xspace}

\renewcommand{\SS}{\ensuremath{\mathbb{S}}\xspace}

\DeclareMathOperator*{\argmin}{arg\,min}

\title{Bi-Level Online Control without Regret}
\author{Andrey Bernstein \thanks{This work was supported by the U.S. Department of Energy under Contract No. DE-AC36-08GO28308 with the National Renewable Energy Laboratory. The U.S. Government retains and the publisher, by accepting the article for publication, acknowledges that the U.S. Government retains a nonexclusive, paid-up, irrevocable, worldwide license to publish or reproduce the published form of this work, or allow others to do so, for U.S. Government purposes.}
}
\begin{document}

\maketitle

\begin{abstract}
This paper considers a bi-level discrete-time control framework with real-time constraints, consisting of several local controllers and a central controller. The objective is to bridge the gap between the online convex optimization and real-time control literature by proposing an online control algorithm with small dynamic regret, which is a natural performance criterion in nonstationary environments related to real-time control problems. We illustrate how the proposed algorithm can be applied to real-time control of power setpoints in an electrical grid.
\end{abstract}

\section{Introduction}
We consider a bi-level discrete-time control framework with real-time constraints, consisting of several \emph{local controllers} and a \emph{central controller}. The tasks of a local controller are (i) to implement setpoints issued by the central controller and (ii) to advertise a prediction of the objective function and the constraints on the feasible setpoints to the central controller. In turn, the central controller uses these advertisements and its system-wide measurements and modeling to compute the next feasible setpoints for the local controllers. This framework is illustrated in Figure \ref{fig:frame}.

This framework is appropriate in the modern real-time control of cyber-physical systems, such as electrical-grid control (e.g., \cite{Jahangiri,Paudyal11, OID, commelec1,opfPursuit}) or the control of autonomous vehicles (e.g., \cite{auton1,auton2,auton3}). In particular, in the context of electrical-grid control, a similar approach was used in \cite{commelec1, commelec2} to control the power setpoints of devices in real time. Moreover, \cite{opfPursuit} considered a distributed control framework in this spirit. However, in \cite{commelec1,commelec2}, a heurisitic was used, and no theoretical guarantees were provided, whereas  \cite{opfPursuit} assumed perfect knowledge of the objective function and feasible sets at the time of the decision making at the central controller.  In \cite{errDiff}, an approach for error correction in local controllers was proposed; however, the performance of the closed-loop system was not analyzed theoretically.

\begin{figure}
  \centering
  \includegraphics[width=0.6\columnwidth]{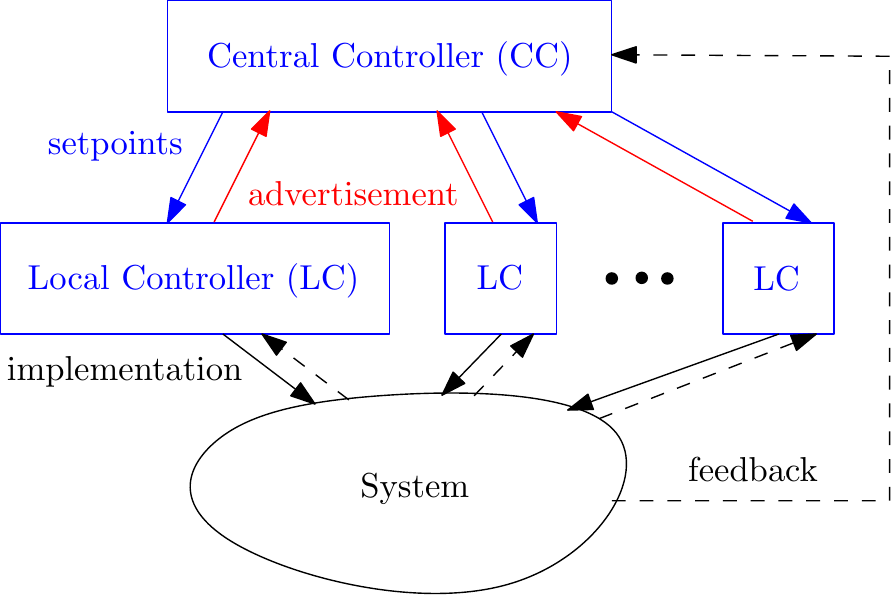}
\caption{Illustration of the bi-level control framework.} \label{fig:frame}
\end{figure}

Observe that this framework is reminiscent of online learning and, in particular, of online convex optimization (OCO) algorithms. Indeed, at each time step, the central controller is faced with a nonstationary (time-varying) optimization problem, and it is required to track its optimal solution. The literature on OCO is focused on devising online algorithms that have provably bounded (or vanishing) \emph{average regret} \cite{black54,zinkevich,CesaBianchi,Hazan2007}; see also \cite{hazanBook} for a recent overview of the subject. However, these  algorithms typically lack the control perspective because they naturally operate in an \emph{open-loop} fashion. In particular, an online optimization algorithm issues setpoints that are assumed to be ``implemented'' perfectly by the system. Hence, contrary to typical control settings, there is no explicit feedback from the system on the actual implementation. 

This paper is the first attempt to bridge the gap between online convex optimization and real-time control. We propose a first-order online control algorithm in the spirit of \cite{zinkevich}, which we call the Online Gradient Control (OGC), in the above-mentioned bi-level framework. We show theoretical guarantees on the algorithm's \emph{dynamic regret} (or \emph{tracking regret}). 
The latter is an extension of the standard regret notion, which compares the performance of an online algorithm to an arbitrary sequences of setpoints (rather than a single fixed setpoint) in hindsight (see, e.g., \cite{Herbster,Adamskiy,hazan2009,dynRegHall}). This regret notion is natural in nonstationary environments associated with real-time control problems. We also discuss the application of the framework and algorithm to the real-time control of heterogeneous devices and, in particular, to the real-time control of power setpoints in an electrical grid.

An additional contribution of this paper is in the context of online convex optimization with \emph{time-varying feasible sets}. Indeed, as the online control problem involves sets of feasible setpoints that change with time, the no-regret result established in this paper is also applicable in the ``open-loop'' optimization setting. To the best of our knowledge,  the only work that considers time-varying feasible sets explicitly is that of  \cite{neu2014}; however, it is assumed there that these sets are drawn from a \emph{fixed} unknown distribution, whereas in the present paper we assume an \emph{arbitrary} sequence of feasible sets.

The paper is organized as follows. Section \ref{sec:frame} presents the bi-level control framework and relevant notation. Section \ref{sec:ogc} introduces our OGC algorithm and analyzes its dynamic regret. Section \ref{sec:app} shows how to apply the algorithm to control a mix of heterogeneous resources, and, in particular, to control an electrical grid. Finally, Section \ref{sec:conc} concludes the paper and outlines some future research directions.

\section{Bi-Level Control Framework} \label{sec:frame}
Assume that there are $J$ local controllers (LCs) indexed by $j = 1, \ldots, J$. The discrete time step index is denoted by $n =1, 2, \ldots$. Let $\S_n(j)$ denote the convex compact set of feasible setpoints of LC $j$ at time step $n$.  Also, let $C_n^{(j)}: \S_n(j) \rightarrow \reals$ denote a convex cost function that represents the objective function of LC $j$.  At each time step $n$, controller $j$ sends to the central controller (CC) an \emph{advertisement} $(\A_{n+1}(j), \hat{C}_{n+1}^{(j)})$ of its feasible set and cost function valid for time step $n+1$ by using \emph{a persistent predictor}, namely $\A_{n+1}(j) = \S_{n}(j)$ and $\hat{C}_{n+1}^{(j)} = C_{n}^{(j)}$. 

Upon receiving the advertisements from all the LCs, the CC computes the feasibility constraints on the overall system based on its system view and the advertisements. Let $\U_{n} \subseteq \A_{n+1} := \S_n(1) \times ... \times \S_n(J)$ denote the compact convex set representing the system feasibility constraints. The CC also computes the estimation of the overall objective function:
\begin{align} \label{eqn:objF}
    F_{n}(x) &:= \sum_{j = 1}^J w_j \hat{C}_{n+1}(x(j)) + G_{n}(x) \nonumber \\
    & = \sum_{j = 1}^J w_j C_{n}(x(j)) + G_{n}(x)
\end{align}
for any $x \in \U_{n}$, where $w_j$ are some weighting (normalization) factors, and the convex function $G_{n}(x)$ represents a system-wide objective. Finally, the CC computes a vector of setpoints $x_{n+1}$ based on $F_{n}(x)$ and $\U_{n}$, and sends the individual setpoints $x_{n+1}(j)$ to the LCs. Upon receiving $x_{n+1}(j)$, LC $j$ implements a feasible approximation $y_{n+1}(j) \in \S_{n+1}(j)$, and the process repeats.
  The interaction between the LCs and the CC is summarized in Algorithm \ref{alg:inter}.

\begin{algorithm}[h!]
\caption{Interaction between local controllers and central controller} \label{alg:inter}
\begin{algorithmic}[1]
\STATE{Set $n = 0$.}
%\FOR{ever}
\LOOP
    \STATE{At time step $n$, every LC $j = 1, ..., J$:} \label{state:n}
    \begin{enumerate}[(a)]
        \item Receives a setpoint request $x_n(j)$ sent by the CC.
        \item Implements an approximation $y_n(j)$ of $x_n(j)$. The implemented setpoint $y_n(j)$ is constrained to lie in the set $\S_n(j)$ representing the local feasibility constraints.
        \item Predicts its  feasible set $\S_{n+1}(j)$  by $\A_{n+1}(j) := \S_n(j)$.
        \item Predicts its local objective function $C_{n+1}^{(j)}$ by $\hat{C}_{n+1}^{(j)} := C_{n}^{(j)}$.
        \item Sends to the CC $\A_{n+1}(j)$ and  $\hat{C}_{n+1}^{(j)}$ over a communication network.
    \end{enumerate}
    \STATE{Upon receiving the advertisements from all the LCs, the CC:}
    \begin{enumerate}[(a)]
        \item Computes the feasibility constraints on the overall system $\U_{n} \subseteq \A_{n+1} := \S_{n}(1) \times ... \times \S_{n}(J)$.
        \item Computes the setpoints' vector $x_{n+1} \in \U_{n}$ based on its current objective function \eqref{eqn:objF}.
        \item Sends $x_{n+1}(j)$ to every LC $j$ over a communication network.
    \end{enumerate}
    \STATE{$n := n + 1$.}
\ENDLOOP
\end{algorithmic}
\end{algorithm}

  \section{Online Gradient Control} \label{sec:ogc}
  In this section, we present our Online Gradient Control (OGC) algorithm. The algorithm is based on the following two steps: 
  \begin{enumerate}[(i)]
      \item The central controller chooses setpoints according to the online gradient descent algorithm:
  \begin{equation} \label{eqn:alg_CC}
      x_{n+1} = \proj[\U_n]{\hat{y}_n - \alpha  \nabla F_n (\hat{y}_n)},
  \end{equation}
  where $\proj[\U]{\cdot}$ is the projection operator, $\alpha$ is a step-size parameter, and $\hat{y}_n$ is the estimation of the setpoint $y_n$ implemented by the LCs at time step $n$. The latter is  obtained either from the LCs or by system-wide measurements.
  \item Upon receiving a setpoint $x_n(j)$, LC $j$ implements a projected version thereof, namely:
  \begin{equation} \label{eqn:alg_LC}
      y_{n}(j) = \proj[\S_n(j)]{x_n(j)}.
  \end{equation}  
  \end{enumerate}

We next analyze the performance of the OGC in terms of optimality and stability.  To this end, we first introduce the following assumptions and definitions.
\begin{assumption} \label{asm:meas}
The measurement of the implemented setpoint $\hat{y}_n$ is $\varepsilon$-accurate. That is, for all $n$,
\[
\|y_n - \hat{y}_n\| \leq \varepsilon.
\]
\end{assumption}

\begin{assumption} \label{asm:lip}
The gradients $\{\nabla F_n\}$ are uniformly bounded and Lipschitz continuous with a common parameter $\lambda < \infty$. Namely, for all $n$, all $x, x' \in \U_n$, $\| \nabla F_n(x) - \nabla F_n(x')\| \leq \lambda \|x - x'\|$. Let $F$ denote a finite constant that is a uniform upper bound on $\|\nabla F_n(x)\|$. 
\end{assumption}

\begin{assumption} \label{asm:bounded_sets}
The sequence of feasible sets $\{\S_n\}$ is uniformly bounded. That is, there exists a finite constant $B$ such that $|\S_n| \leq B$ for all $n$. Let $D$ denote the upper bound on the diameters of $\{\S_n\}$, so that $\diam (\S_n) \leq D$ for all $n$.
\end{assumption}

  \begin{definition}[Admissible sequence of setpoints]
  A sequence $\bx_{1:n} := \{x_k\}_{k=1}^n$ is said to be \emph{admissible} if $x_k \in \U_k$ for every $k = 1, ..., n$.
  Let $\X_n$ denote the set of all the admissible sequences of length $n$, and let $\X_{\infty}$ denote the set of  all the infinite admissible sequences $\bx_{1:\infty}$.
  \end{definition}
  
  \begin{definition}[Dynamic regret] \label{def:regret}
  Consider the sequence of implemented setpoint $\by_{1:n} := \{y_k\}_{k=1}^n$ by the control algorithm up to time step $n$. The \emph{total dynamic regret} of the  algorithm with respect to a given sequence $\bz_{1:n} \in \X_n$ is defined as
  \begin{equation}
  r_n(\bz_{1:n}) := \sum_{k=1}^n (F_k(y_k) - F_k(z_k)).
  \end{equation}
  Similarly, $r_n(\bz_{1:n})/n$ is the \emph{average dynamic regret}.
  \end{definition}
    
  \begin{definition}[Temporal variability]
  For any sequence $\bx_{1:n} = \{x_k\}_{k = 1}^n$, let
  \begin{equation}
      V(\bx_{1:n}) := \sum_{k = 1}^n \|x_k - x_{k+1}\|
  \end{equation}
  denote its temporal variability.
  \end{definition}

  \begin{thm} \label{thm:ogc_conv}
  Under Assumptions \ref{asm:meas}, \ref{asm:lip}, and \ref{asm:bounded_sets}, for any $\bz_{1:\infty} \in \X_{\infty}$ and any $\alpha > 0$, we have that
   \begin{equation} \label{eqn:reg_bound}
   \limsup_{n \rightarrow \infty} \frac{r_n(\bz_{1:n})}{n} %\left(\frac{1}{n} \sum_{k=1}^n F_k(y_k) - \frac{1}{n} \sum_{k=1}^n F_k(z_k) \right) 
   \leq K_1 \alpha + \frac{K_2(1 + \alpha \lambda)}{\alpha}\varepsilon + \frac{K_3}{\alpha} \limsup_{n \rightarrow \infty} \frac{V(\bz_{1:n})}{n} 
   \end{equation}
   where
   \[
   K_1 := \frac{F^2}{2}, \, K_2 := \frac{[2(D +  \alpha F) + (1 + \alpha \lambda)\varepsilon]}{2}, \, \text{and } \, K_3 := D + B.
   \] 
   In particular, \eqref{eqn:reg_bound} is valid for the optimal sequence 
   \[
   \bz^*_{1:n} \in \argmin_{\bz_{1:n} \in \X_n} \sum_{k=1}^n F_k(z_k).
   \]
  \end{thm}

  \begin{proof}
  The proof follows that of \cite[Theorem 1]{zinkevich}. For simplicity of exposition, we use below a scalar-style notation; however, the proof works for vectors by interpreting the regular multiplication as the inner product. 
  
  Let $\bz_{1:n+1} \in \X_{n+1}$. We have that
  \begin{align}
    &(y_{n+1} - z_{n+1})^2 \leq (x_{n+1} -  z_{n+1})^2 \nonumber \\
    &= (x_{n+1} - z_n +  z_n -  z_{n+1})^2 \nonumber \\
    &= (x_{n+1} - z_n)^2 + 2 (x_{n+1} - z_n) (z_n -  z_{n+1}) + (z_n -  z_{n+1})^2 \nonumber \\
    &\leq (\hat{y}_n - z_n - \alpha \nabla F_n (\hat{y}_n))^2 \nonumber \\
    &\quad + [2 (x_{n+1} - z_n) + (z_n -  z_{n+1})](z_n -  z_{n+1}) \nonumber\\
 &\leq (\hat{y}_n - z_n - \alpha \nabla F_n (\hat{y}_n))^2 + 2(D + B) \|z_n -  z_{n+1}\| \label{eqn:proof_step1}, 
  \end{align}
  where the first inequality follows by using \eqref{eqn:alg_LC}, the fact that $z_{n+1} \in \U_{n+1} \subseteq \S_{n+1} $, and the non-expansive property of the projection operator; the second inequality holds by \eqref{eqn:alg_CC}, the fact that $z_n \in \U_n$,   and the non-expansive property of the projection operator; and in the last inequality, we used the Cauchy-Schwarz inequality and the fact that under Assumption \ref{asm:bounded_sets}
  \begin{align*}
  \|2 (x_{n+1} - z_n) + (z_n -  z_{n+1})\| &
  \leq 2 \|x_{n+1} - z_n\| + \|z_n\| + \|z_{n+1}\| \\ 
  &\leq 2 \diam (\S_n) + |\S_n| + |\S_{n+1}|\\
  &\leq 2 (D + B).
  \end{align*}
We now expand the first term in \eqref{eqn:proof_step1}. It holds that
\begin{align}
& (\hat{y}_n - z_n - \alpha \nabla F_n (\hat{y}_n))^2 \nonumber \\
& = (y_n - z_n - \alpha \nabla F_n (y_n) + (\hat{y}_n - y_n) + \alpha (\nabla F_n (\hat{y}_n) - \nabla F_n (y_n)))^2. \label{eqn:proof_step2}
\end{align}
Let
\[
\gamma_n := (\hat{y}_n - y_n) + \alpha (\nabla F_n (\hat{y}_n) - \nabla F_n (y_n))
\]
and note that under Assumptions \ref{asm:meas} and \ref{asm:lip}, we have
\[
\|\gamma_n\| \leq (1 + \alpha \lambda) \varepsilon.
\]
Continuing the derivation in \eqref{eqn:proof_step2}, we obtain 
\begin{align}
&(\hat{y}_n - z_n - \alpha \nabla F_n (\hat{y}_n))^2 \nonumber \\
    &\leq (y_n - z_n - \alpha \nabla F_n (y_n))^2\nonumber \\
    &\quad + \left[2(y_n - z_n - \alpha \nabla F_n (y_n)) + \gamma_n  \right] \gamma_n \nonumber\\
&\leq (y_n - z_n - \alpha \nabla F_n (y_n))^2 \nonumber\\
&\quad + [2(D +  \alpha F) + (1 + \alpha \lambda)\varepsilon] (1 + \alpha \lambda)\varepsilon \nonumber\\
    &= (y_n - z_n)^2 - 2 \alpha \nabla F_n (y_n)(y_n - z_n) + \alpha^2 (\nabla F_n (y_n))^2\nonumber \\
    &\quad + [2(D +  \alpha F) + (1 + \alpha \lambda)\varepsilon] (1 + \alpha \lambda)\varepsilon \nonumber\\
    &\leq (y_n - z_n)^2 - 2 \alpha (F_n (y_n) - F_n(z_n)) + F^2 \alpha^2 \nonumber\\
  &\quad + [2(D +  \alpha F) + (1 + \alpha \lambda)\varepsilon] (1 + \alpha \lambda)\varepsilon  \label{eqn:proof_step3}
  \end{align}
  where the second inequality holds by the Cauchy-Schwarz inequality and Assumptions \ref{asm:lip} and \ref{asm:bounded_sets}; and the last inequality holds by the standard argument for comparing the instantaneous regret of linear and strictly convex functions (see, e.g., \cite{zinkevich}). Combining \eqref{eqn:proof_step1} and \eqref{eqn:proof_step3} yields
   \begin{align}
    &(y_{n+1} - z_{n+1})^2  \nonumber \\
&\leq (y_n - z_n)^2 - 2 \alpha (F_n (y_n) - F_n(z_n)) + F^2 \alpha^2 \label{eqn:proof_step4}\\
  &\quad + [2(D +  \alpha F) + (1 + \alpha \lambda)\varepsilon] (1 + \alpha \lambda)\varepsilon \nonumber \\
&\quad + 2(D + B) \|z_n -  z_{n+1}\| \nonumber
  \end{align}
  
  By rearranging \eqref{eqn:proof_step4}, we obtain
  \begin{align*}
  &F_n (y_n) - F_n(z_n) \leq [ (y_n - z_n)^2 - (y_{n+1} - z_{n+1})^2 ]/(2 \alpha) + \alpha F^2/2 \\
  &\quad + [2(D +  \alpha F) + (1 + \alpha \lambda)\varepsilon] (1 + \alpha \lambda)\varepsilon/(2 \alpha) +  (D + B) \|z_n -  z_{n+1}\|/\alpha.
  \end{align*}
  Averaging the last inequality yields
  \begin{eqnarray*}
  \frac{1}{n} \sum_{k=1}^n (F_k(y_k) - F_k(z_k)) &\leq& [ (y_1 - z_1)^2 - (y_{n+1} - z_{n+1})^2 ]/(2 \alpha n)\\
  &&\quad+ K_1 \alpha + \frac{K_2(1 + \alpha \lambda)}{\alpha}\varepsilon + (K_3/\alpha) \frac{1}{n} V(\bz_{1:n}),
  \end{eqnarray*}
  which completes the proof.
  \end{proof}
  
  \iffalse
  \begin{thm}
  If $\alpha_n = 1/\sqrt{n}$, we have that
   \[
   \limsup_{n \rightarrow \infty} \left(\frac{1}{n} \sum_{k=1}^n F_k(y_k) - \frac{1}{n} \min_{\bar{z}_n \in \X_n} \sum_{k=1}^n F_k(z_k) \right) \leq \frac{K_3 V(\bar{z}^*_n)}{\sqrt{n}}
   \]
   for some $K_3 < \infty$ and
   \[
   \bar{z}^*_n \in \argmin_{bar{z}_n \in \X_n} \sum_{k=1}^n F_k(z_k).
   \]
  \end{thm}
  \fi
  
  Observe that Theorem \ref{thm:ogc_conv} establishes that if the measurement error $\varepsilon$ is small, and if the optimal trajectory $\bz^*_{1:n}$ varies slowly (or infrequently), in the sense that 
  \[
  \limsup_{n \rightarrow \infty} \frac{V(\bz^*_{1:n})}{n} 
  \]
 is small, then the corresponding average dynamic regret will be small for the appropriately chosen step-size $\alpha$.

\begin{rem}
The discrete-time control problem considered here is an approximation to the corresponding continuous-time optimal control problem.  The result of Theorem \ref{thm:ogc_conv} thus implicitly states that if the optimal continuous-time trajectory $z^*(t)$ is continuous in $t$, then one can choose a fine enough discretization of the timescale  so that the average time variability  $V(\bz^*_{1:n})/n$ is asymptotically small, hence yielding small asymptotic regret.
\end{rem}

  \begin{rem}
Note that when $\hat{y}_n = x_n$ and the feasible set $\U_n$ does not depend on $n$, the algorithm \eqref{eqn:alg_CC} is the well-known online gradient descent algorithm first introduced in \cite{zinkevich}. The case when $\hat{y}_n = x_n$ but $\U_n$ \emph{depends} on  $n$ can be considered as an (open-loop) online optimization setting with time-varying feasible sets, and the result of Theorem \ref{thm:ogc_conv} applies to this case as well.  
  \end{rem}

  We conclude this section by noting  that the OGC algorithm is stable in the input-to-state stability sense by construction. To that end, consider the nonlinear dynamical system for the state variable $y_n$ given by
  \begin{equation}
  y_{n+1} = \proj[\S_{n+1}]{ \proj[\U_n]{y_n - \alpha  \nabla F_n (y_n) + \epsilon_n}},
  \end{equation}
  where $\epsilon_n$ is the measurement error associated with Assumptions \ref{asm:meas} and \ref{asm:lip}.
  Here, the sets $\{\S_n\}$ (and the sequence $\{\epsilon_n\}$) can be considered as \emph{exogenous inputs} to this dynamical system. 
  %\begin{prop} \label{prop:stab}
  Under Assumption \ref{asm:bounded_sets}, it is clear that $\|y_n\| \leq B$ for all $n$,
  %\end{prop}
  %\begin{proof}
  %The result trivially follows by noting that $y_n \in \S_n$ for all $n$.
  %\end{proof}
%  \begin{rem}
  %We note that Proposition \ref{prop:stab} 
  which establishes a bounded-input-bounded-state (BIBS) stability. Indeed, it states that if the sequence of ``inputs'' $\{\S_n\}$ is uniformly bounded, so is the sequence of ``states'' $\{y_n\}$.
%  \end{rem}
  
  %\section{Tracking bounds}
  %Use the fact that
  %\[
  %|G(\bar{y}_n) - G(\bar{z}^*_n)| \geq \alpha \|\bar{y}_n - \bar{z}^*_n\|_2^2
  %\]
  %for strongly convex (?) functions.

\section{Application to Real-Time Control of Heterogeneous Devices} \label{sec:app}

In this section, we consider the setting where the LCs control heterogeneous devices of two general types: (i) \emph{convex devices} with convex feasible sets and (ii) \emph{discrete} devices with discrete  feasible sets with a \emph{finite} number of elements. Observe that for convex devices, the OGC can be directly applied. For discrete devices, we propose the following randomized scheme in the spirit of repeated games. In Section \ref{sec:rogc}, we give the general algorithm, whereas in Section \ref{sec:grids}, we outline the application in the context of electrical-grid control.

\subsection{Randomized Online Gradient Control (ROGC)} \label{sec:rogc}
Let $\cJ_{c} \bigcup \cJ_{d} = \{1, \ldots, J\}$ denote the partition of the devices into convex and discrete ones, respectively. Note that for $j \in \cJ_d$, the feasible set $\S_n(j)$ is discrete, hence non-convex. %and $C_n^{(j)}: \cS_n(j) \rightarrow \reals$ denote the objective function of LC $j$  at time step $n$. 

The ROGC algorithm is exactly the same as the OGC algorithm for the CC (cf.~\eqref{eqn:alg_CC}) and every LC $j \in \cJ_c$ (cf.~\eqref{eqn:alg_LC}). On the other hand, each LC $j \in \cJ_d$ performs the following:
\begin{enumerate}[(i)]
\item Advertise:
\begin{align} \label{eqn:advDisc1}
\A_{n+1}(j) &:= \Delta (\S_n(j)) \\
\hat{C}_{n+1}^{(j)}(y) &:= \sum_{s \in \S_n^{(j)}} C_n^{(j)}(s) y(s), \quad y \in \A_{n+1}(j),\label{eqn:advDisc2}
\end{align}
where $\Delta (\S)$ is the probability simplex imposed by a discrete set $\S$. Observe that $\A_{n+1}(j)$ is a convex set, and $\hat{C}_{n+1}^{(j)}(y)$ is the expected value of $C_n^{(j)}$ with respect to $y \in \Delta (\S_n^{(j)})$, thus a linear function of $y$. 
\item Compute:
  \begin{align} 
      y_{n}(j) = \proj[\Delta(\S_n(j))]{x_n(j)}.
  \end{align}  
\item Implement \emph{a random control} $S_n(j) \in \S_n(j)$ drawn from a probability distribution $y_{n}(j)$.
\end{enumerate}

Let 
\begin{align}
R_n(\bz_{1:n}) := \sum_{k = 1}^n \left( \left[ \sum_{j \in \cJ_c} w_j C_k^{(j)}(y_k(j)) \right ] + \left [\sum_{j \in \cJ_d} w_j C_k^{(j)}(S_k(j)) \right] + G_k(y_k) - F_k(z_k)  \right)
\end{align}
denote the dynamic regret of the ROGC algorithm at time step $n$ with respect to $\bz_{1:n}$ (cf.~Definition \ref{def:regret}). Note that $R_n(\bz_{1:n})$ is a random variable due to appearance of the random variables $\S_n(j)$. The following result is a direct application of Theorem \ref{thm:ogc_conv} to the ROGC algorithm.

\begin{cor}
Under the conditions of Theorem \ref{thm:ogc_conv}, the expected regret of the ROGC algorithm is bounded by
\begin{align*}
\limsup_{n \rightarrow \infty} \mathbb{E} \left( \frac{R_n(\bz_{1:n}) }{n} \right) \leq  K_1 \alpha + \frac{K_2(1 + \alpha \lambda)}{\alpha}\varepsilon + \frac{K_3}{\alpha} \limsup_{n \rightarrow \infty} \frac{V(\bz_{1:n})}{n} 
\end{align*}
for any $\bz_{1:\infty} \in \X_\infty$.
\end{cor}

\begin{rem}
In our setting, we implicitly assume that the environment is \emph{oblivious} in the sense that it produces the same sequence of objective functions and feasible sets regardless of the applied control actions. The high-probability regret bounds can thus be obtained similarly to that shown in \cite[Lemma 1]{zinkevich}.  The extension to non-oblivious environments is a subject for future work.
\end{rem}

\begin{rem} \label{rem:nonConv}
Note that when the feasible set of a discrete device has large cardinality, the proposed ROGC algorithm might be impractical as it will require manipulations of large vectors. 
However, the ROGC algorithm  can be extended to cover this case (or any case of \emph{non-convex} bounded feasible set $\S_n(j)$) if  instead of considering directly the probability simplex $\Delta(\S_n(j))$ as the advertised convex feasible set, one considers a convex hull $\conv (\S_n(j))$.  The idea is to identify every $y \in \conv (\S_n(j))$ with a probability distribution $p_y \in \Delta(\S_n(j))$ parametrized by $y$ such that $y = \mathbb{E}_{S \sim p_y}(S)$. The extension to this case is a subject of ongoing work. 
\end{rem}

\subsection{Application to Real-Time Control of Electrical Grids} \label{sec:grids}
Consider the problem of controlling a collection of heterogeneous electrical resources that are interconnected via a portion of a power grid (e.g., a \emph{distribution feeder} or a \emph{microgrid}) under a typically time-varying objective and certain safety constraints. These resources can be photovoltaic (PV) systems, wind power plants, batteries, buildings, and electric vehicles. 
The resources are typically connected to the network via power inverters, thus allowing for the direct control of the (active and reactive) power setpoints at the point of connection.
This problem has recently received renewed interest through the advent of renewable energy sources, such as solar power, and improved battery and inverter technologies. 

%We focus on real-time control, namely on second (or even sub-second) time scale, due to high volatility of renewable energy sources (solar, wind).

The network-wide objective of the CC is to keep the power grid within the operational constraints (e.g., keeping node voltages and line currents within limits). Another possible goal of the CC is  to ensure that the power flow at the point of connection to the higher-level grid follows a given time-varying signal -- namely, making this subgrid \emph{dispatchable}. We next outline a concrete real-time control problem in the spirit of \cite{opfPursuit, commelec1}.

Consider an electrical distribution system comprising  $J+1$ nodes collected in the set $\cJ  \cup \{0\}$, $\cJ := \{1, \ldots, J\}$. Node $0$ is defined to be the distribution substation at which the voltage is fixed. Let $V_j \in \comps$ denote the voltage phasor at node $j = 1,\ldots J$, and let $v :=  [|V_1|, \ldots, |V_J|]^\sfT \in \mathbb{R}^{J}$ denote the vector collecting the voltage magnitudes. Without the loss of generality, we assume that there is a resource connected at every node $j \in \cJ$, thus identifying node $j$ with LC $j$. The control variable for each resource $j$ is given by $x(j) = (P(j), Q(j))^\sfT \in \reals^2$, where $P(J)$ and $Q(j)$ are the active and reactive power setpoints, respectively. As a convention, positive power means production, whereas negative power signifies consumption. Let $x = (x(1)^\sfT, \ldots, x(J)^\sfT)^\sfT \in \reals^{2J}$ collect the setpoints of the $J$ resources.  The relationship between $v$ and $x$ is given by the well-known nonlinear alternating-current (AC) power-flow equations $f(v, x) = 0$ (see, e.g., \cite{Vittalbook}). 

To design the controllers, the nonlinear power-flow equations $f$ are typically convexified or linearized around the current operation point. For the purpose of the example here, consider a possibly time-varying linear approximation to $f$ in the form
\begin{equation} \label{eqn:V}
v = A_n x + a_n,
\end{equation}
where the system-dependent matrix $A_n \in \mathbb{R}^{2J \times J}$ and vector $a_n \in \mathbb{R}^{J}$ can be computed in a variety of ways (e.g., ~\cite{Baran89,christ2013sens,sairaj2015linear,swaroop2015linear,bolognani2015linear} and pertinent references therein). Similarly, the active power flow at the substation $P(0)$ (namely, the power that is exported to the higher-level grid) can be approximated as
\begin{equation} \label{eqn:P0}
P(0) = w_n^\sfT x + b_n,
\end{equation}
for some $w_n \in \reals^{2J}$ and $b_n \in \reals$.

\subsubsection{Design of the Central Controller}

The CC obtains the advertisements $(\A_{n+1}(j), \hat{C}_{n+1}(j))$ from LCs $j\in \cJ$; see Section \ref{sec:LC_grid} below for details on  how the LCs compute those. It also receives the estimation of the implemented  power setpoint $\hat{y}_n$.  It then constructs the objective function $F_n$ according to \eqref{eqn:objF}. The network-wide objective $G_n$ is designed using \eqref{eqn:P0} to track a given sequence of power setpoints $\{P_{0, n}^{set}\}$ at the substation:
\begin{equation}
G_n(x) = 0.5 \left(w_n^\sfT x + b_n -  P_{0, n}^{set}\right)^2.
\end{equation}
The network-wide feasibility constraints are constructed using \eqref{eqn:V} as
\begin{equation}
\U_n := \left \{x \in \A_{n+1}: V_{\min} \leq  (A_n x + a_n)_j \leq V_{\max}, \, j \in \cJ \right\}
\end{equation}
which ensures that the voltage magnitudes are within the prescribed limits $V_{\min}$ and $V_{\max}$. Finally, the control variables for the next time step are computed using \eqref{eqn:alg_CC}.

\subsubsection{Design of Local Controllers} \label{sec:LC_grid}

For  the purpose of this example, suppose that every resource is either (i) a PV system; (ii) a heating, ventilating, and air-conditioning (HVAC) system;  or (iii) a battery.
These three types of devices cover most modern distributed energy resources. Indeed, the PV system represents a volatile renewable power generator, the HVAC system represents a non-convex (discrete) controllable load, and the battery represents a bi-directional energy-storage resource.
For every type of resource, we next present typical cost functions and feasibility constraints that are used to construct the advertisements to the CC.

For a PV system with inverter's rated power $S_{\text{inv}}(j)$ and an available active power $P_{\textrm{av},n}(j)$, the set $\S_n(j)$ is given by 
$$ \S_n(j)  =  \left\{(P(j), Q(j)) \hspace{-.1cm} : \,  0 \leq {P}(j) \leq  P_{\textrm{av},n}(j), \, {Q}(j)^2  \leq  S_{\text{inv}}(j)^2 - {P}(j)^2 \right\}; $$
see, e.g., \cite{commelec2,opfPursuit}.
Note that for PV inverters, the set $\S_n(j)$ is convex, compact, and time varying (it depends on the available power $P_{\textrm{av},n}(j)$, which in turn depends on the solar irradiance). The associated cost function typically encourages active-power generation and penalizes reactive power, e.g., 
\begin{equation} \label{eqn:pv_cost}
C_n^{(j)}(P(j), Q(j)) = -c_1 P(j) + c_2 Q(j)^2 
\end{equation}
for some positive constants $c_1$ and $c_2$.  Note that a PV system is a convex resource in the terminology of Section \ref{sec:rogc}, and thus the advertisement is defined as $\A_{n+1}(j) =  \S_n(j)$ and $\hat{C}_{n+1}^{(j)} = C_n^{(j)}$.

Consider now a simplified case of an HVAC system that can  be in either the ON or OFF state.  When the system is in the ON state, it consumes  $P_{\max}(j)$ active power and $0$ reactive power. Moreover, the system can be \emph{locked} in either the ON or OFF state because of cycling limitations and other constraints. Let $\ell_n(j) \in \{0, 1\}$ denote a binary state variable that equals $1$ if the system is locked at time step $n$. The details of the related state machine are omitted as they are not essential for this example. With this at hand, the finite set of feasible setpoints $\S_n(j)$ is given by
\[
\S_n(j) = \begin{cases}
\{0, -P_{\max}(j)\},  & \text{ if } \ell_n(j) = 0, \\
\{0\}, & \text{ if } \ell_n(j) = 1, \, P_{n-1}(j) = 0 \\ 
\{-P_{\max}(j)\}, & \text{ if } \ell_n(j) = 1, \, P_{n-1}(j) = -P_{\max}(j). 
\end{cases}
\]
Finally,  the cost of  being in the ON state, $C_n^{(j)}(-P_{\max}(j))$, and in the OFF state, $C_n^{(j)}(0)$, is system-dependent and reflects, for example,  the current temperature and its distance from the desired setting.  Observe that an HVAC system is a discrete resource in the terminology of  Section \ref{sec:rogc}, and therefore the advertisement is defined by \eqref{eqn:advDisc1} and \eqref{eqn:advDisc2}.  In particular, let $y \in [0, 1]$ denote the probability to turn the HVAC system on. Then
\[
\A_{n+1}(j) = \begin{cases}
[0, 1],  & \text{ if } \ell_n(j) = 0, \\
\{0\}, & \text{ if } \ell_n(j) = 1, \, P_{n-1}(j) = 0 \\ 
\{1\}, & \text{ if } \ell_n(j) = 1, \, P_{n-1}(j) = -P_{\max}(j)
\end{cases}
\]
and for $y \in \A_{n+1}(j)$, $\hat{C}_{n+1}^{(j)}(y) =  (1 - y)C_n^{(j)}(0) + yC_n^{(j)}(-P_{\max}(j))$.

Finally, consider a battery with state-of-charge at time step $n$ given by $\text{SoC}_n \in [0, 1]$.  Let $P_{\min, n}(j)$ and $P_{\max, n}(j)$ denote, respectively, the lower and upper bounds on active power production. These are time-varying quantities that depend on operating conditions of the battery, such as $\text{SoC}_n$ and the DC voltage; see, e.g., \cite{commelec2}.  With inverter's rated power $S_{\text{inv}}(j)$, the set $\S_n(j)$ is given by 
$$ \S_n(j)  =  \left\{(P(j), Q(j)) \hspace{-.1cm} : \,   P_{\min, n}(j) \leq {P}(j) \leq  P_{\max, n}(j), \, {Q}(j)^2  \leq  S_{\text{inv}}(j)^2 - {P}(j)^2 \right\} $$ 
similarly to the PV system. The associated cost function can be designed based on the $\text{SoC}_n$ and the desired value for the state-of-charge. For example, if $\text{SoC}_n$ is greater than the desired value, a function that encourages power production can be defined, similarly to \eqref{eqn:pv_cost}. Conversely, if $\text{SoC}_n$ is smaller than the desired value, a function that encourages power consumption can be defined, e.g.:
\[
C_n^{(j)}(P(j), Q(j)) = c_1 P(j) + c_2 Q(j)^2 
\]
for positive $c_1$ and $c_2$.  Note that, similarly to the PV system,  a battery is a convex resource, thus $\A_{n+1}(j) =  \S_n(j)$ and $\hat{C}_{n+1}^{(j)} = C_n^{(j)}$.

To conclude this section, we note that in \cite{commelec1} a special case of the OGC algorithm was used as an heuristic to control a realistic power system (a microgrid), and it was shown that this algorithm performed well numerically.

\section{Conclusion and Future Work} \label{sec:conc}

We  presented and analyzed a first-order control algorithm that bridges the gap between online convex optimization and  real-time control. We  showed that this algorithm possesses small dynamic regret under certain conditions on the measurement error and time variability of the optimal trajectory. The algorithm can be applied to control heterogeneous resources in real time, and in particular to control a mix of convex and discrete resources. We also illustrated the application of the algorithm in the context of real-time control of the electrical grid.

The proposed OGC algorithm is only a first (and the most straightforward)  example of control algorithms that can be applied in the proposed framework. An extension to other methods to optimize the setpoints is of interest. In particular, it is an interesting question whether a distributed algorithm (e.g., based on primal-dual decomposition method as in \cite{opfPursuit} or on the alternating direction method of multipliers (ADMM))  can be analyzed similarly to show the small dynamic regret. Further, extending the framework to general non-convex feasible sets seems possible (cf.~Remark \ref{rem:nonConv}) and is a subject of ongoing work.
Finally, an interesting research direction is in devising online control algorithms based on the concept of \emph{approachability} \cite{black54}, which is a more general concept than no-regret algorithms.

\bibliographystyle{plain}
\bibliography{biblio}
\end{document}